\documentclass[11pt,reqno]{amsart}

\newcommand\version{August 26, 2026}

\usepackage{amsmath,amsfonts,amsthm,amssymb,amsxtra}
\usepackage{stmaryrd}
\usepackage{bbm} 
\usepackage{hyperref}	
\usepackage{mathrsfs}
\usepackage{esint}

\usepackage{xfrac}
\usepackage{todonotes, relsize, pgfplots, tikz, float}
\usetikzlibrary{calc,decorations.markings}
\tikzset{
	> = stealth,
	every pin/.style = {pin edge = {}},
	flow/.style = {decoration = {markings, mark=at position #1 with {\arrow{>}}},
		postaction = {decorate}
	},
	flow/.default = 0.5,
	main/.style = {color=#1, line width=0.5pt, line cap=round, line join=round},
	main/.default = black,
	fontscale/.style={font=\relsize{#1}},
}

\newtheorem{theorem}{Theorem}[section]
\newtheorem{proposition}[theorem]{Proposition}
\newtheorem{lemma}[theorem]{Lemma}
\newtheorem{corollary}[theorem]{Corollary}

\theoremstyle{definition}

\theoremstyle{remark}

\newtheorem{remark}[theorem]{Remark}

\numberwithin{equation}{section}

\numberwithin{equation}{section}

\renewcommand{\epsilon}{\varepsilon}

\newcommand{\N}{\mathbb{N}}

\renewcommand{\phi}{\varphi}
\newcommand{\R}{\mathbb{R}}

\newcommand{\Z}{\mathbb{Z}}

\DeclareMathOperator{\supp}{supp}

\DeclareMathOperator{\M}{\mathbf M}
\DeclareMathOperator{\I}{\mathbf I}
\DeclareMathOperator{\F}{\operatorname{Fill}}
\DeclareMathOperator{\FV}{\operatorname{FV}}
\DeclareMathOperator{\lip}{\operatorname{lip}}

\usepackage{mathtools, todonotes, thmtools}\usepackage{tikz}
\usetikzlibrary{datavisualization}
\usetikzlibrary{datavisualization.formats.functions}
\usetikzlibrary{patterns}

\let\oldtocsection=\tocsection

\let\oldtocsubsection=\tocsubsection

\let\oldtocsubsubsection=\tocsubsubsection

\renewcommand{\tocsection}[2]{\hspace{0em}\oldtocsection{#1}{#2}}
\renewcommand{\tocsubsection}[2]{\hspace{1em}\oldtocsubsection{#1}{#2}}
\renewcommand{\tocsubsubsection}[2]{\hspace{2em}\oldtocsubsubsection{#1}{#2}}

\begin{document}

\begin{titlepage}
    \huge \title[Linear isoperimetric filling inequalities]{Linear isoperimetric filling inequalities in Hadamard spaces at and above the asymptotic rank}
		\vspace{7cm}
\end{titlepage}

\subjclass[2020]{Primary: 51F30. Secondary: 49Q15, 30L15, 53C23.}
\keywords{$\mathrm{CAT}(0)$-spaces, linear filling inequalities, isoperimetric gap, asymptotic rank, asymptotic Nagata dimension, integral currents, Federer--Fleming deformation}
\date{\version}
\thanks{\copyright\, 2026 by the author. This paper may be reproduced, in its entirety, for non-commercial purposes.}

\author{Jonas W. Peteranderl}
\address[Jonas W.~ Peteranderl]{Mathematisches Institut, Ludwig-Maximilians-Universit\"at M\"unchen, The\-resienstr.~39, 80333 M\"unchen, Germany}
\email{peterand@math.lmu.de}

\begin{abstract}
    Reformulated in terms of the asymptotic rank, a conjecture by Gromov predicts a linear isoperimetric filling inequality in all dimensions greater than or equal to the asymptotic rank of a Hadamard space, in contrast to the Euclidean-type nonlinear behavior below this threshold. We prove the predicted linear inequality for Hadamard spaces with finite asymptotic Nagata dimension and finite asymptotic rank. More precisely, every integral cycle of dimension at or above the asymptotic rank admits a filling whose mass is bounded linearly in the mass of the cycle. Our proof is based on a new self-improvement mechanism for Wenger's sub-Euclidean growth theorem. This approach upgrades the asymptotic rank-one result by Wenger and the recent asymptotic rank-two result by Lang, Stadler, and Urech from exponents arbitrarily close to one to the optimal linear exponent. Moreover, the result extends to arbitrary finite asymptotic ranks.
\end{abstract}

\maketitle
\setcounter{page}{1}


\section{Introduction and main results}

Isoperimetric inequalities and their applications arise in a remarkably broad variety of geometric and analytic contexts. They range from classical Euclidean and Riemannian settings \cite{Oss78,Gro83,Yau75} to geometric measure theory and general metric spaces \cite{FF60,Gro83,Wen05}. More recently, isoperimetric-type inequalities also appear in metric measure spaces with synthetic curvature bounds \cite{CM17} and in Lorentzian geometry \cite{BE99,CM25,LP25} with a reversed sign. There are also many classical applications of isoperimetric inequalities in other areas of mathematics such as discrete geometry \cite{Dod84}, spectral geometry \cite{Che70, Bus82}, and geometric group theory \cite{Gro93}, for instance.

  The standard Euclidean isoperimetric inequality gives a lower bound on the boundary area of a domain $\Omega\subseteq \R^{k+1}$ by the $k/(k+1)$-th power of its volume. Similarly, Gromov \cite[\S3.4.C]{Gro83} proved the Euclidean-type isoperimetric filling inequality for cycles in Hadamard manifolds. In terms of the Ambrosio--Kirchheim integral currents \cite{AK00}, it says that every integral current $T$ with $\partial T=0$ has a filling $V$ with
  $$\M(V)\leq C(k) \M(T)^{1+\frac1k}$$ up to some constant $C(k)>0$, where the mass $\M$ denotes the current-theoretic generalization of the volume. So far, the optimizers of this inequality are known to be Euclidean $(k+1)$-balls in a few cases, notably in Euclidean space \cite{Alm86} and $k=2$ \cite{Sch20}.

 More generally, we are concerned with \textit{$\mathrm{CAT}(0)$-spaces}, which are geodesic metric spaces with triangles that are not ``thicker'' than Euclidean comparison triangles. 
A complete $\mathrm{CAT}(0)$-space is also known as \textit{Hadamard space}. A natural question is whether on a Hadamard space $X$ the exponent on the right side of the corresponding isoperimetric inequality is optimal. Gromov's isoperimetric gap conjecture \cite{Gro93} predicts that the Euclidean exponent $1+1/k$ jumps to the linear exponent $1$ when the dimension grows and that this dimensional transition occurs exactly at the ``rank'' of $X$; see also \cite{LW08}. While Gromov suggested several different almost equivalent notions of rank, the unified notion of \textit{asymptotic rank} is nowadays prevailing and supported by several advances in the field, most notably by Kleiner \cite{Kle99} and Wenger \cite{Wen11}.

 Complete answers to the conjecture are known for symmetric spaces of noncompact type in terms of Lipschitz chains by Leuzinger \cite[Theorem~1, (ii)]{Leu14} (see also Gromov's assertion in \cite{Gro93}) and homogeneous Hadamard manifolds in terms of currents by Isleifsson \cite[Theorem 1.1]{Isl25}. On general Hadamard spaces, Wenger \cite{Wen11}
 proved that the Euclidean growth of filling functions is sharp below the asymptotic rank and established sub-Euclidean growth at and above it, though without any quantitative rate.

Wenger \cite[Corollary 1.3]{Wen11} further proved that the asymptotic rank is at most $1$ if and only if $X$ is Gromov hyperbolic. As a consequence, linear growth occurs for $k=1$ (cf.~\cite[p.~106]{Gro87} and \cite{Wen08}). For $k>1$, the growth rate of the filling function was only known to be arbitrarily close to linear; cf.~\cite[Theorem 1.6]{Wen11}. This implies that for a $(k+1)$-filling $V$ of a compactly supported integral $k$-cycle $T$, $k\geq 2$, the $(1+\delta)$-th power, $0<\delta<1/k$, of the mass of $T$ bounds the mass of $V$ from above, which can be written as \begin{equation}\label{eq:delta_isop}
    \M(V)\leq C(\delta)\M(T)^{1+\delta}\,,
\end{equation}
with a $\delta$-dependent constant $C(\delta)=C(\delta,k,X)>0$. In the context of word-hyperbolic groups, higher-dimensional linear isoperimetric inequalities were found earlier; see \cite{Lan00}.

For general proper metric spaces satisfying suitable cone-type inequalities, Goldhirsch--Lang \cite[Theorem~5.3]{GL23} proved linear filling inequalities for cycles with uniformly controlled density and constants depending on this control. For proper $\mathrm{CAT}(0)$-spaces with asymptotic rank at most~$2$, Dru\c{t}u--Lang--Papasoglu--Stadler \cite{DLPS25} proved homotopical isoperimetric inequalities with exponent $1+\delta$ for Lipschitz $2$-spheres and surfaces of higher (fixed) genus. Subsequently, Lang--Stadler--Urech \cite{LSU26} established the corresponding homological estimate \eqref{eq:delta_isop} without properness, but under the condition that the \textit{asymptotic Nagata dimension} is finite to compensate for the missing topological control.

In this work, we show that the exponent $\delta$ can be avoided by discretizing and deforming a filling one dimension higher than the initial cycle. As a byproduct of our approach, the restriction to asymptotic rank at most $2$ can be dropped.

 To state our main results, we fix some notation, which we discuss in Section~\ref{sec:2} and \ref{sec:3} in more detail. In the following, $\I_k(X)$ denotes the integral $k$-currents of finite mass and $\I_{k,c}(X)$ the subgroup with compact support. Further, we write $$\F_X(T)\coloneqq \inf\{\M(V): V\in \I_{k+1}(X),\partial V =T\}$$ for the filling volume of $T\in \I_k(X)$ with $\partial T=0$ and $$\FV_{k+1}^X(r)\coloneqq \sup\{\F_X(T):T\in \I_k(X), \partial T=0,\M(T)\leq r\}$$ for the filling volume function in $r\geq 0$. 
 If $T\in \I_{k,c}(X)$, the infimum in $\F_X(T)$ is attained by a (compactly supported) filling $V$; see Theorem~\ref{thm:wen_ex}.

\begin{theorem}[Linear self-improvement]\label{thm:1} Let $X$ be a complete $\mathrm{CAT}(0)$-space of finite asymptotic Nagata dimension, and let $k\geq 1$. If 
\begin{equation}\label{eq:subEucl}
    \limsup_{r\to \infty} \frac{\FV_{k+1}^X(r)}{r^{\frac{k+1}k}}=0\,, 
\end{equation}then there is a constant $C=C(X,k)>0$ such that every $T\in \I_{k}(X)$ with $\partial T=0$ has a $V\in \I_{k+1}(X)$ satisfying
$$\partial V=T\qquad\text{and}\qquad\M(V)\leq C \M (T)\,.$$
If $T\in \I_{k,c}(X)$, then it is possible to choose $V\in \I_{k+1,c}(X)$.\end{theorem}

Note that $X$ is not necessarily proper. Combining the above theorem with Wenger's theorem \cite{Wen11} on the sub-Euclidean growth for $k\geq \nu$ with finite asymptotic rank $\nu$ gives the higher rank linear isoperimetric filling inequality.

\begin{theorem}[Linear filling inequality]\label{thm:2} 
    Let $X$ be a complete  $\mathrm{CAT}(0)$-space with finite asymptotic Nagata dimension and finite asymptotic rank $\nu$. For every integer $k\geq\max\{ \nu,1\}$, there is a $C=C(X,k)>0$ such that every $T\in \I_{k}(X)$ with $\partial T=0$ satisfies
    \begin{equation}
        \F_X(T)\leq C \M(T)\,.\label{eq:linfil}
    \end{equation}If $T\in \I_{k,c}(X)$, then the support of the filling realizing $\F_X(T)$ can be chosen to be compact.
\end{theorem}

 In the special case $\nu\leq 2$ and $k\geq 2$, Theorem \ref{thm:2} is the $\delta$-free limit of the results in \cite{Wen11} and \cite{LSU26}. As we shall see, our proof is not restricted to asymptotic rank at most $2$ as it does not rely on the uniform density bounds \cite[Lemma 4.2]{LSU26}, which are based on ideas from \cite{Sta21}. Moreover, due to an application of Wenger's thick-thin decomposition \cite[Theorem~4.2]{Wen11a}, we can state our theorems more generally for currents of finite mass, as opposed to the compactly supported setting in \cite{LSU26}.
 
 The linear exponent in \eqref{eq:linfil} is (asymptotically) sharp. This can already be seen by considering growing geodesic balls in real hyperbolic space. Since our assumptions agree with the ones in \cite{LSU26} for asymptotic rank~at most~$2$, Theorem \ref{thm:2} applies to all instances they describe, including Hada\-mard $3$-manifolds, finite-dimensional $\mathrm{CAT}(0)$-cube complexes, and suitably coverable Gromov hyperbolic $\mathrm{CAT}(0)$-spaces; see \cite{LSU26} for more details. 
 
Theorem~\ref{thm:1} and \ref{thm:2} fit into the program of unifying higher-rank hyperbolicity initiated by Kleiner--Lang \cite{KL20} and further developed by Goldhirsch--Lang \cite{GL23}. In both of our theorems, we allow cycles to be arbitrary integral currents of finite mass and assume apart from that only conditions on the ambient space $X$, which makes it possible to phrase the linear filling inequality as the linear asymptotics of the filling function \begin{equation}\label{eq:FV_asymp}
    \FV_{k+1}^X(r)=\mathcal O(r)\,.
\end{equation}

This immediately yields a characterization of the asymptotic rank in terms of the filling function. Other notable consequences of Theorem \ref{thm:2} are the positivity of Cheeger's constant for oriented Hadamard manifolds with finite asymptotic Nagata dimension and finite asymptotic rank and an improved deformation theorem of \cite[Theorem~5.1]{BWY23} without support-dependent constants and without assuming a linear structure. To continue with the proof of our main results, we postpone those consequences to Section~\ref{sec:6}.

Our proof of Theorem \ref{thm:1}, and thus Theorem \ref{thm:2}, is based on the following proposition, which decomposes a compactly supported cycle $T$ into a current $Q$ with absorbable filling term and a boundary $\partial S$ whose filling is mass-dominated by $T$. 

\begin{proposition}[Decomposition with absorbable filling term]\label{prp:1} Let $X$ be a complete $\mathrm{CAT}(0)$-space with finite asymptotic Nagata dimension at most $n$, and let $k\geq 1$.
     There are constants $a_0,b_0, c_0, s_0>0$ such that for every $T\in \I_{k,c}(X)$ with $\partial T=0$ and every $s\geq s_0$ there exists $S\in \I_{k+1,c}(X)$ and $Q\in \I_{k,c}(X)$ with 
    \begin{equation}\label{eq:decomp1}
        T=Q+\partial S\,,
    \end{equation} 
    \begin{equation}\label{eq:decomp2}
        \M(S)\leq a_0 s \M(T)\,,
    \end{equation} and a filling $W\in \I_{k+1,c}(X)$ of $Q$ satisfying
    \begin{equation}\label{eq:decomp3}
        \M(W)\leq b_0 \frac{\FV_{k+1}^X(c_0 s^k)}{s^{k+1}}  \F_X(T)\,.
    \end{equation}
\end{proposition}

Note that all constants $a_0,b_0,c_0,s_0$ depend only on $k$ and the quantitative data from the covers in the definition of the finite asymptotic Nagata dimension; see Remark \ref{rmk} and Subsection \ref{subsec:2.5}. In particular, they do not depend on any current, simplex number, or scale.

\subsection{Proof strategy and comparison with Lang--Stadler--Urech}\label{subsec:proof}

The proof of Proposition \ref{prp:1} consists of five steps. 

In Step 1 we choose a minimizing filling $V$ of $T$ and, at a sufficiently large scale $s$, apply \cite[Lemma 2.1]{LSU26} to the support of the filling $V$ to pass to a finite-dimensional simplicial complex $\Sigma$. This is the only place where the finite asymptotic Nagata dimension enters the proof; for a more detailed discussion see Subsection \ref{subsec:2.5}. In Step 2 we push $V$ to a current $V'$ in $\Sigma$ and deform $V'$ to obtain $$V'=P+R+\partial D$$ with a simplicial $(k+1)$-chain $P$ and the remaining terms having controlled mass. Taking boundaries eliminates the term $\partial D$. In Step 3 we write $P=\Sigma_\sigma n_\sigma \llbracket \sigma\rrbracket$, $n_\sigma\in \Z$, and give an upper bound on the number of top-dimensional simplices $$\sum_\sigma |n_\sigma|\lesssim \frac1{s^{k+1}}\F_X(T)\,.$$ 
In Step 4 we push the resulting boundary decomposition $\partial V'=\partial P+\partial R$ back to $X$ and obtain the currents $Q$ and $S$ satisfying \eqref{eq:decomp1} and \eqref{eq:decomp2}. Finally, in Step 5 we fill the boundaries of each of the $(k+1)$ simplices in $X$. Together with the estimate on the simplex count, this yields the absorbable bound~\eqref{eq:decomp3}.

Compared to \cite{LSU26}, the main difference is that we deform a $(k+1)$-dimensional minimizing filling at a single scale after discretization instead of decomposing $k$-dimensional cycles iteratively on multiple scales. In \cite[Theorem 5.3]{LSU26} the corresponding simplicial count is of the form $$\sum_\sigma |n_\sigma|\lesssim \frac1{s^{k}}\M(T)\,,$$ so in contrast we acquired one power of $s^{-1}$. This is the main analytic gain that facilitates the final absorption argument. There is also an algebraic gain in the proof as the deformation produces a boundary term $\partial D$ that disappears when taking boundaries, so its mass does not enter the final estimate. Moreover, our approach avoids the controlled density bounds \cite[Lemma 4.2]{LSU26} that were used in the rank-two argument of \cite{LSU26}, and thus extends beyond the rank-two threshold.

Nevertheless, there is a hidden trade in the covering hypothesis. Our global assumption is a finite asymptotic Nagata dimension. However, we only use it when applying \cite[Lemma 2.1]{LSU26} in form of a uniform quantitative covering of $(k+1)$-dimensional minimizing filling supports $V$, whereas \cite{LSU26} applies its covering argument to the supports of the $k$-dimensional cycles occurring in their iteration.
On the other hand, our deformation is performed only once, at one sufficiently large scale, rather than through a multiscale iteration. In this regard, our approach with the absorption in one step gives a simple, quantitative dependence of the covering scale, the sub-Euclidean growth rate of the filling function, and the isoperimetric constant.

Passing to higher dimensions to make a difficult boundary problem more tractable appears also in other fields. For instance, there is a structural analogy between deforming a filling and surgery for cobordisms (see \cite{Mil65}, for instance). Another example is the Caffarelli--Silvestre extension \cite{CS07}, which represents a nonlocal operator on $\R^d$ as the boundary of a local operator on $\R^{d+1}$. Although it is not of the same scope, we hope that our strategy described above will be useful to prove other linear filling inequalities in the future.

\subsection{Structure of the paper} In Section \ref{sec:2} we recall basic notions such as  integral currents, fillings, asymptotic rank, and asymptotic Nagata dimension from general metric theory including classical results. After these preliminaries, we discuss the two main technical inputs, an approximation by simplicial complexes and a Federer--Fleming-type deformation theorem, in Section \ref{sec:3}. The mathematical core of our paper is Section \ref{sec:4}, where Proposition \ref{prp:1} is proved. By choosing a suitable scale when applying this proposition in Section \ref{sec:5}, we deduce the linear self-improvement theorem, Theorem~\ref{thm:1}, which together with Wenger's sub-Euclidean growth theorem establishes the linear isoperimetric filling inequality, Theorem \ref{thm:2}. To pass from compactly supported currents to currents of finite mass, we rely on a thick-thin decomposition, which is discussed in Appendix \ref{app}. Finally, in Section~\ref{sec:6} we provide some direct consequences of our linear filling inequality.

\subsection*{Acknowledgments} The author would like to thank Stephan Stadler for his valuable comments at an early stage of this project. Partial support through the German Research Foundation grants FR 2664/3-1  and TRR 352-Project-ID 470903074 and through the Studienstiftung des deutschen Volkes is acknowledged.

\section{Preliminaries}\label{sec:2}
 
Here and in the following, we will always work in a $\mathrm{CAT}(0)$-space $X$ with metric $d$. Elementary examples of $\mathrm{CAT}(0)$-spaces include, for instance, Euclidean and Hilbert spaces, Hadamard manifolds, and $\R$-trees. $\mathrm{CAT}(0)$-spaces have many useful properties, most notably for our later purposes, the existence and uniqueness of geodesics and Busemann convexity. We refer to \cite{Bal95,BH99,BBI01} for more details on $\mathrm{CAT}(0)$-spaces.

\subsection{Integral currents}\label{subsec:2.1}

We work in the theoretical framework of metric integral currents in a complete metric space $X$; see \cite{AK00,Lan11}. In contrast to \cite{LSU26}, we do not restrict our discussion to compactly supported currents. We will give a short exposition of metric current theory, aligning with its use in our proofs, and refer to the above mentioned works for more details. 

Let $\mathrm{Lip}(X)$ be the space of real-valued Lipschitz functions and $\mathrm{Lip}_b(X)$ be the subspace of bounded functions in $\mathrm{Lip}(X)$. A $k$-dimensional \textit{current} $T$, $k\geq 0$, is a multilinear functional of the form $T:\mathrm{Lip}_b(X)\times \mathrm{Lip}(X)^k\to \R$, satisfying some additional locality and continuity assumptions; cf.~\cite{AK00}. If $k\geq 1$, the \textit{boundary} of the $k$-dimensional current $T$ is the $(k-1)$-dimensional current $\partial T$ defined by $$\partial T(g_{0}, \dots, g_{k-1})\coloneqq T(1,g_0,\dots, g_{k-1})\,.$$ It satisfies $\partial^2 T=0$. A current with $\partial T=0$ is called a \textit{cycle}. If $k=0$, then we call $T$ a cycle if $T(1)=0$ in accordance with \cite{BWY23}. An \textit{integral current} is integer rectifiable with a finite mass boundary; cf.~\cite{AK00}. We denote by $\I_{*}(X)$ the chain complex of finite-mass integral currents and by $\I_{*,c}(X)$ the chain subcomplex of compactly supported integral currents.

There is a natural associated Borel measure $\|T\|$ on $X$ that has finite total \textit{mass} $\M(T)\coloneqq\|T\|(X)$ and is compactly supported if and only if $T$ is. The mass $\M$ induces a metric on $\I_{k}(X)$, $k\geq 0$, given by $(S,T)\mapsto \M(S-T)$ for $S,T\in \I_{k}(X)$.

For a Lipschitz continuous function $h:X\to Y$ with $Y$ some other complete metric space, we define the \textit{pushforward} of $T$ by $$h_\#T(g_0,\dots,g_k)\coloneqq T(g_0\circ h, \dots, g_k\circ h)\,.$$ The Lipschitz constant of $h$ at $x\in X$ is denoted by $\lip h(x)$. A useful inequality for Lipschitz functions $h$ is
\begin{equation} \label{eq:lip_gen}
   \|h_\# T\|(B)\leq \int_{h^{-1}(B)} \lip (h)(x)^k \,\mathrm d\|T\|(x)
\end{equation} for every Borel set $B\subseteq Y$; see \cite[Lemma 3.3]{BWY23}, for instance. If $\lip h(x)\leq L$, then \begin{equation}
\label{eq:lip_bdd}\M(h_\# T)\leq L^k \M(T)\,.
\end{equation}

 If $E\subset \R^k$ is a measurable set with finite measure, this induces a current~$\llbracket E\rrbracket$ defined~by \begin{equation}
     \llbracket E\rrbracket (g_0,\dots, g_k)\coloneqq \int_{E} g_0 \det D(g_1,\dots, g_k) \,\mathrm dx\,.\label{eq:E_form}
 \end{equation} Here $g_i$, $i\in \{0,\dots,k\}$, denote Lipschitz functions, whose partial derivatives exist by Rademacher's theorem.
Note that $\llbracket E\rrbracket$ is the classical current of integration over measurable $E\subseteq \R^k$; it acts on smooth differential $k$-forms $g_0 \,\mathrm d g_1\wedge \dots \wedge \mathrm d g_k$ on~$E$.

Another example, which is relevant for our argument, is provided by currents associated with simplicial chains. Let $\Sigma$ be a finite (Euclidean) simplicial complex. Fix one orientation for each $k$-simplex $\sigma\subseteq \Sigma$. If we denote by $\llbracket\sigma\rrbracket\in \I_{k,c}(\Sigma)$ the current of integration over the oriented $k$-simplex $\sigma$, then an \textit{integral simplicial $k$-chain} can be uniquely represented by the finite sum
$$P=\sum_\sigma n_\sigma \llbracket\sigma\rrbracket\,,\qquad n_\sigma\in \Z\,,$$
and we can identify the simplicial chain with the corresponding integral current.

In the next subsection we discuss products of currents with intervals in the context of homotopy currents.

\subsection{Geodesic homotopy currents}\label{subsec:2.3}
Let $T\in \I_k(X)$ be a cycle and $h:\supp T\to X$ an $L$-Lipschitz function. We construct a (geodesic) homotopy $$H:[0,1]\times \supp T\to X\,,\qquad H(t,x)=[x,h(x)](t)\,,$$ where $[x,y](t)$ denotes the (unique) constant-speed geodesic between $x$ and~$y$. The product current $\llbracket 0,1 \rrbracket \times T\in \I_{k+1}([0,1]\times X)$ is defined as the multilinear functional 
\begin{align*}(\llbracket 0,1 &\rrbracket \times T)(g_0,\dots ,g_{k+1})\\&\coloneqq \sum_{i=1}^{k+1}(-1)^{i+1} \int_0^1 T(g_{0t}\partial_tg_{it}, g_{1t},\dots ,g_{(i-1)t}, g_{(i+1)t},\dots ,g_{(k+1)t})\,\mathrm dt\,,\end{align*} where $g_{jt}(x)=g_j(t,x)$ for $t\in [0,1]$, $x\in X$, and $j\in\{0,\dots,k+1\}$;
 see \cite[Section~3.3]{BWY23}, for instance. Since $\partial T=0$, the boundary formula \cite[Proposition~3.4]{BWY23} implies \begin{equation}\label{eq:hash2}
    \partial H_\# (\llbracket0,1\rrbracket\times T)= (H_1)_\# T-(H_0)_\# T= h_\# T-T\,.
\end{equation}We write $H_t(x)=H_x(t)=H(t,x)$. For fixed $x$, constant speed gives $$\lip  H_x(t)\equiv d(x, h(x))\,.$$ For fixed $t$, Busemann convexity \cite{Bus55} applied to the two geodesics from $x$ to $h(x)$ and $y$ to $h(y)$ gives $$d(H(x,t),H(y,t))\leq (1-t) d(x,y)+t d(h(x),h(y))\,,\qquad x,y\in X\,.$$ Hence, we obtain
$$\lip  H_t(x)\leq (1-t)+t L \leq \max\{1, L\}\,.$$ 
Combining both Lipschitz bounds, the bound in \cite[Lemma 3.5]{BWY23} implies the useful mass bound \begin{align}
    \M(H_{\#}(\llbracket0,1\rrbracket\times T))&\leq (k+1)\int_0^1 \int_X \lip H_x(t) \lip H_t(x)^k\,\mathrm d\|T\|(x)\mathrm d t\notag\\&\leq(k+1)\sup_{x\in \supp T}d(x, h(x))(\max\{1, L\})^k\M(T)\,. \label{eq:mass1}
\end{align}
If $T$ does not have compact support, the right side could be infinite.

\subsection{Fillings and related theorems}

Let $T\in \I_{k}(X)$, $k\geq 1$, be a cycle. Then we call a current $V\in \I_{k+1}(X)$ a \textit{filling} of $T$ if $\partial V=T$. Fillings of boundedly supported cycles satisfy a \textit{cone-type inequality} for $\I_{k}(X)$.

\begin{theorem}[{\cite[Theorem 4.1~($\kappa=0$)]{Wen06}}]\label{thm:wen_coning}Let $X$ be a complete $\mathrm{CAT}(0)$-space and $k\geq 1$. For every $T\in \I_{k}(X)$ with $\partial T=0$ and supported in a closed ball of radius $r>0$, there exists a filling $V\in \I_{k+1}(X)$ with \begin{equation}
    \label{eq:cone_type}
\M(V)\leq \frac{r}{k+1}\M(T)\,.\end{equation}
\end{theorem}
 
The filling $V$ with the minimal mass determines the \textit{filling volume} of a cycle $T\in \I_{k}(X)$, $k\geq 1$, which is given by $$\F_X(T)\coloneqq \inf \{\M(V):V\in \I_{k+1} (X) \,\text{and}\, \partial V=T\}\,.$$ For $T\in \I_{k,c}(X)$ in a Hadamard space $X$, this infimum is attained by a compactly supported filling. Indeed, the following theorem holds.
 
\begin{theorem}[{\cite[Theorem 1.6]{Wen05}}]\label{thm:wen_ex} Let $X$ be a complete $\mathrm{CAT}(0)$-space and $k\geq 1$. For every $T\in \I_{k,c}(X)$ with $\partial T=0$, there exists a $V\in \I_{k+1,c}(X)$ with \begin{equation}\label{eq:min_fill}
    \partial V=T\qquad \text{and}\qquad \M(V)=\F_X(T)\,,
\end{equation} and every $V\in \I_{k+1}(X)$ satisfying \eqref{eq:min_fill}
has compact support.
\end{theorem} 

Hence, when dealing with compactly supported currents $T$, the notion of $\F_X(T)$ coincides with the infimum taken over all compactly supported fillings $V$.

\subsection{Asymptotic rank and asymptotic Nagata dimension}\label{subsec:2.5}
In this subsection we define and discuss the notions of asymptotic rank, asymptotic Nagata dimension, and filling functions. A description of the use of the Nagata condition here and in \cite{LSU26} is included.

The asymptotic rank as a unified notion of rank in the large was coined by Wenger \cite{Wen11}. For Hadamard spaces, it has the following equivalent definition. The Hadamard space $X$ has \textit{asymptotic rank} at least $m\geq 1$ if there is a sequence $(r_i)_{i\in \N}$ with $r_i>0$ and $r_i\to \infty$ as $i\to \infty$ and subsets $A_i\subseteq X$ such that the rescaled sets $(A_i, d/r_i)$ converge to the unit ball $B^m\subseteq \R^m$ in the Gromov--Hausdorff sense. This agrees with the Euclidean rank in case $X$ is proper and cocompact; see \cite[Theorem 3.4]{Wen11}, for instance. We adapt the convention from \cite[Proposition 3.1]{Wen11} that the asymptotic rank is $0$ if no such positive $m$ exists, which is the case if and only if $X$ is bounded; see \cite[p.~248]{Wen11}, for instance.

The asymptotic Nagata dimension makes the topological dimension in terms of coverings \textit{quantitative} by adding an additional control on the diameter and multiplicity of the covering sets. To describe it, we require a few more notions. A family $\mathscr B=(B_i)_{i\in I}$ of subsets in $X$ is called \textit{$D$-bounded}, $D>0$, if every $B_i$ has diameter at most $D$, and $\mathscr B$ has \textit{$s$-multiplicity} at most $n$ if every set in $X$ of diameter not larger than $s$ intersects no more than $n$ members of the family. We say that $X$ has \textit{linearly controlled asymptotic dimension} at most $n$ or \textit{asymptotic Nagata dimension} at most $n$ if there are $s_0,c>0$ such that for every $s>s_0$, the space $X$ admits a $cs$-bounded cover of $s$-multiplicity at most $n+1$. Therefore, the asymptotic Nagata dimension is encoded in a cover with data $(n,c,s_0)$. If the definition holds not only for $s>s_0$ but for all $s>0$, then $X$ has \textit{Nagata dimension} at most $n$. This notion was introduced by Assouad \cite{Ass82} and systematically studied by~\cite{LS05}.

 The covers of $X$ associated with an asymptotic Nagata dimension at most~$n$ can be assumed to be open (up to a change of the constant $c$); see \cite[Proposition 1.7]{DS07}. Therefore, after restricting such an open cover to a compact subset $K$, we obtain a finite-cover formulation with the same multiplicity. Indeed, we find $s_0, c_1>0$ such that for every $s\geq s_0$ there is a finite open $c_1s$-bounded cover of $K$ with $s$-multiplicity at most $n+1$.

 Assuming finite asymptotic Nagata dimension of $X$ is a convenient way to guarantee suitable controlled covers of compact supports of cycles in $X$ and sufficiently large scales~$s$. Scrutinizing \cite[Theorem 5.1, Theorem 5.3]{LSU26}, it appears that only a finite covering property for the cycle support at a specific scale $s$ is used. In contrast, their iterative decomposition in \cite[Section 6]{LSU26} requires a quantitative covering on multiple scales $s_1,s_2, s_3,\dots$ with a \textit{uniform} covering dimension $n$ and constant $c$. Since our approach does not require a multiscale decomposition, it suffices to work at one common scale, though in our case the full filling supports and not merely the cycle supports have to be covered.

 Finally, in terms of the asymptotic rank $\nu$, we state Wenger's sub-Euclidean growth result for the filling volume function for dimensions $k\geq \nu$. Recall that the filling volume function for $r>0$ is defined by $$\FV_{k+1}^X(r)\coloneqq \sup\{\F_X(T): T\in \I_{k}(X), \partial T=0, \M(T)\leq r\}\,.$$ 

Before stating the theorem, we note that a subset $K$ of a metric space~$X$ is called \textit{$q$-quasiconvex}, $q\geq 1$, if all $x,y\in K$ can be connected by a curve in $K$ of length at most $q d(x,y)$. We say that $X$ is \textit{quasiconvex} if it is $q$-quasiconvex for some $q\geq 1$.

\begin{theorem}[{\cite[Theorem 1.2]{Wen11}}] \label{thm:wen_subeuc}Let $X$ be a complete, quasiconvex metric space of asymptotic rank $\nu$ admitting cone-type inequalities \eqref{eq:cone_type} for boundedly supported $m$-cycles with $m\in \{1,\dots,k\}$. If $k\geq \max\{\nu,1\}$, then 
    \begin{equation*}
    \limsup_{r\to \infty} \frac{\FV_{k+1}^X(r)}{r^{\frac{k+1}k}}=0\,.
\end{equation*}
\end{theorem}
Put differently, $X$ satisfies a sub-Euclidean isoperimetric inequality. Subject to the assumptions in Theorem \ref{thm:1}, we upgrade this sub-Euclidean growth to a linear one. 

\section{Discretization and deformation}\label{sec:3}

In this section, we present the main two ingredients for our proof of Proposition \ref{prp:1}. This includes a discussion of their metric compatibility.

\subsection{Passing to a simplicial complex}
The following approximation step from \cite{LSU26} will be substantial for the proof and is also the bottleneck for further potential improvements on the Nagata assumption in Theorem \ref{thm:1}; see also \cite[Proposition 6.1]{BWY23}.

\begin{lemma}[{\cite[Lemma 2.1]{LSU26}}]\label{lem:trans}
    Let $X$ be a complete $\mathrm{CAT}(0)$-space. Assume that there are constants $c,s>0$, an integer $n$, and $Y\subseteq X$ a subset with a finite $cs$-bounded covering and $s$-multiplicity at most $n+1$. Then there exists a finite simplicial complex $(\Sigma, d_2)$ of dimension at most $n$, realized as a subcomplex of some simplex of edge length~$s$ in Euclidean space, and maps $$\psi:Y\to \Sigma\qquad \text{and}\qquad  \phi:\Sigma\to X$$ such that for a constant $L=L(n,c)$
    \begin{enumerate}
        \item[(1)] $\psi$ is $L$-Lipschitz continuous,
        \item[(2)]  $\phi$ is $L$-Lipschitz continuous on each simplex,
        \item[(3)] $f\coloneqq \phi\circ\psi$ is $L$-Lipschitz continuous and $d(y,f(y))\leq Ls$ for $y\in Y$.
    \end{enumerate}
Here $d_2$ denotes the Euclidean metric of the ambient space.
    
\end{lemma}

\subsection{Ambient vs.~length metric}
\label{subsec:3.2} Above we use the ambient metric $d_2$ on~$\Sigma$. However, to apply the deformation theorem below on the full simplicial complex $\Sigma$, as opposed to \cite{LSU26}, we need to take currents to the intrinsic length metric space. 

Assume that $\Sigma$ is path-connected. Besides the extrinsic metric, we can equip $\Sigma$ with the intrinsic length metric $d_{\ell}$ induced by the Euclidean metric. Let $$\iota: (\Sigma,d_2)\to (\Sigma,d_{\ell})$$ be the identity. It always holds that $d_2\leq d_{\ell}$ on $\Sigma$ with equality on each simplex of $\Sigma$. By \cite[Lemma 2.1]{BWY23}, both metrics $d_2$ and $d_\ell$ induce the same topology on~$\Sigma$, and for every $x\in \Sigma$ there is a sufficiently small $d_2$-neighborhood of $x$ such that every sufficiently close $y$ lies with $x$ in a common simplex, and thus $d_2(x,y)=d_\ell(x,y)$ in this case. Hence, $\iota$ satisfies $\lip \iota (x)\leq 1$ for each $x\in \Sigma$, while it is an isometry when restricted to any simplex and Lipschitz continuous by compactness. Moreover, \eqref{eq:lip_gen} implies for a current $A$ that \begin{equation}
    \label{eq:ident_mass}\M(\iota_\# A)\leq \M(A)\qquad \text{and} \qquad \M(\iota_\# (\partial A))\leq \M(\partial A)\,.
\end{equation}
Note that $\phi$ is geodesic on $(\Sigma,d_\ell)$, and thus Lipschitz continuous by \cite[Lemma~2.4]{BWY23}.
 
If $\Sigma$ is not path-connected, all metric changes and deformations that are about to follow are performed componentwise and the resulting currents are summed by linearity. This holds in particular for the proof of Proposition~\ref{prp:1} in Section \ref{sec:4}. Therefore, we suppress an additional index for the component if there is no confusion.

Recall that if $\Sigma$ is the disjoint union of $N$ path-connected components $\Sigma_i$, $i\in \{1,\dots,N\}$, then we can split a current and its boundary componentwise
\begin{equation}
    \label{eq:split}T=\sum_{i=1}^N T\llcorner{\Sigma_i}\qquad \text{and}\qquad \partial T=\sum_{i=1}^N(\partial T)\llcorner{\Sigma_i}\,.
\end{equation}
Note that $\partial (T\llcorner\Sigma_i)=(\partial T)\llcorner{\Sigma_i}$ follows from \cite[Lemma 3.5]{Lan11}. Indeed, the finitely many path-connected components are open and compact, so the characteristic function of $\Sigma_i$ is both Lipschitz continuous and locally constant, and the desired formula follows from \cite[Lemma 3.5]{Lan11}. Note that the mass measure splits additively as well. 

\subsection{Relative deformations} \label{subsec:3.3} In \cite{BWY23}, Basso--Wenger--Young establish a Federer--Fleming-type deformation theorem (cf.~\cite{FF60,Fed69}) on general metric spaces under some convexity and triangulation assumptions.

We say that a metric space $X$ is \textit{$(n,D,\varepsilon)$-triangulated}, $D\geq 1$, $n\in \N$, $\varepsilon>0$, if
there exist an $n$-dimensional simplicial complex $\Sigma$ and a homeomorphism $p:\Sigma \to X$, which is $D$-bi-Lipschitz on every simplex. Here $\Sigma$ is equipped with a scaled Euclidean metric such that each simplex has side length $\varepsilon$. The map $p$ identifies integral simplicial chains in $\Sigma$ with their  associated integral currents in $X$.

Now we can state the theorem.

\begin{theorem}[{\cite[Theorem A.2, Remark A.6]{BWY23}}]
   \label{thm:deform} Let $X$ be a $q$-quasi\-convex, $(n,D,\varepsilon)$-triangulated metric space. Then there is a constant $C=C(q,n,D)>0$ such that for every $k\in \{1,\dots,n\}$ and every $T\in \I_{k}(X)$ there is an integral simplicial $k$-chain $P$ and currents $R\in \I_{k}(X)$ and $S\in \I_{k+1}(X)$
    such that $$T=P+R+\partial S\,.$$ Moreover, we have
 \begin{alignat}{2}
     \M (P)&\leq C\M(T)\,,
     &\qquad  \M(\partial P)
     &\leq C\M(\partial T)\,,\notag\\
      \M (S)&\leq \varepsilon C\M(T)\,,&\qquad \M (R)&\leq \varepsilon C\M(\partial T)\,. 
      \label{eq:deformest}
 \end{alignat}
 If $k>n$, then $\I_k(X)=0$.
\end{theorem}

For applying Theorem \ref{thm:deform} in the next section, we explain why it can be applied componentwise (if necessary). The situation is particularly simple for the simplicial complex $\Sigma$ from Lemma \ref{lem:trans}. Indeed, under the identification $\iota$, $\Sigma$ is geodesic with respect to $d_\ell$ on each path-connected component, and thus $1$-quasiconvex. Moreover, for $n_i\coloneqq \dim \Sigma_i$, each path-connected component $(\Sigma_i,d_\ell|_{\Sigma_i})$ is $(n_i,1,s)$-triangulated with homeomorphism $p_i\coloneqq \iota|_{\Sigma_i}$. Thus, assuming $C$ grows in $n$ without loss of generality, the deformation constant $C(1,n,1)>0$ can be chosen uniformly in $i$ as $n_i\leq n$. 
 
\section{Proof of the proposition}\label{sec:4}

In this section, we give a proof of the crucial Proposition \ref{prp:1} in five consecutive steps, as described in Subsection \ref{subsec:proof}. 

\textit{Step 1:} Let $T\in \I_{k,c}(X)$ be a cycle. By Theorem \ref{thm:wen_ex}, there exists a minimal filling $V\in \I_{k+1,c}(X)$, that is, $$\partial V=T \qquad \text{and} \qquad \M(V)=\F_X(T)\,.$$ 
    
Without loss of generality, we may assume that $\F_X(T)>0$. Otherwise, $V=0$ as a current and hence $T=\partial V=0$, so taking $Q=S=W=0$ satisfies the assertion. 
    
    By compactness of $\supp V$ and the open cover formulation for the asymptotic Nagata dimension at most $n$, the discussion in Subsection \ref{subsec:2.5} provides $s_0, c_1>0$ such that for every $s\geq s_0$ there is a finite open $c_1s$-bounded cover of $\supp V$ with $s$-multiplicity at most~$n+1$. Applying Lemma \ref{lem:trans}, we obtain a finite simplicial complex $\Sigma$ of dimension at most~$n$, a constant $L=L(n,c_1)$, and $L$-Lipschitz functions $\psi:\supp V\to \Sigma$ and $\phi:\Sigma\to X$ (simplexwise) such that $\phi\circ\psi$ is $L$-Lipschitz continuous, that is, $d(x, \phi\circ\psi(x))\leq Ls$ for all $x\in \supp V$.

 Recall that whenever $\Sigma$ is not path-connected, all computations on $\Sigma$ are carried out on each path-connected component and then summed according to \eqref{eq:split}. 

    \textit{Step 2:} Here we deform the filling $V$ after pushing it to $\Sigma$. We define $$V'\coloneqq \iota_\#\psi_\# V\in \I_{k+1,c}(\Sigma,d_\ell) \quad \text{and} \quad T'\coloneqq \partial V'= \iota_\#\psi_\# T\in \I_{k,c}(\Sigma,d_\ell)\,.$$ As $\psi$ maps $L$-Lipschitz continuously into $(\Sigma,d_2)$ and the mass does not increase via $\iota$ as given in \eqref{eq:ident_mass}, the Lipschitz mass bound \eqref{eq:lip_bdd} implies \begin{equation}\label{eq:preimage_est}
        \M(V')\leq L^{k+1} \F_X(T)\qquad \text{and}\qquad \M(T')\leq L^k \M(T)\,.
    \end{equation}

 If $k+1\leq \dim\Sigma$, Theorem \ref{thm:deform} can be applied since each component of $\Sigma$ is a $1$-quasiconvex, $(n,1,s)$-triangulated metric space; cf.~Subsection \ref{subsec:3.2} and \ref{subsec:3.3}. This provides a decomposition $$V'=P+R+\partial D$$ into an integral simplicial $(k+1)$-chain $P$, a current $R\in \I_{k+1,c}(\Sigma)$, and the boundary current of $D\in \I_{k+2,c}(\Sigma)$. This result also provides the estimates~\eqref{eq:deformest}. We will make use of two of them in the form
 \begin{equation}
     \M (P)\leq C_1\M(V')\qquad\text{and}
     \qquad \M (R)\leq C_1s\M(T')\,, 
      \label{eq:deformest_appl}
 \end{equation}
 for some constant  $C_1=C_1(n)=C(1,n,1)>0$, which is uniform over all path-connected components of $\Sigma$; cf.~Subsection \ref{subsec:3.3}.
 
 As $\partial^2 D=0$, we have \begin{equation*}
     T'=\partial P+\partial R\,.
 \end{equation*} If $k+1\leq \dim\Sigma$ does not hold, then $V'=0$ since a $(k+1)$-dimensional integral current cannot be supported on a finite complex of dimension less than $k+1$ (cf.~last part of Theorem \ref{thm:deform}), and thus $T'=\partial V'=0$, so $P=R=D=0$ gives a trivial decomposition.

\textit{Step 3:} Next, we count the top-dimensional simplices. Write the finite simplicial chain~$P$ (uniquely) as $$P=\sum_\sigma n_\sigma \llbracket \sigma \rrbracket\,, \qquad n_\sigma\in \Z\,,$$ where we sum over all $(k+1)$-simplices of the finite simplicial complex $\Sigma$ after fixing one orientation for each simplex. If $\omega_{k+1}$ denotes the Euclidean $(k+1)$-volume of the regular, unit-edge $(k+1)$-simplex, then every $(k+1)$-simplex of $\Sigma$ has (intrinsic) volume $\omega_{k+1}s^{k+1}$. Since distinct open simplices are disjoint, we have \begin{equation}\label{eq:sumnsigma}
    \omega_{k+1}s^{k+1} \sum_\sigma |n_\sigma|=\M(P)\leq C_1L^{k+1}\F_X(T)\,,
\end{equation} where we used the first inequalities in \eqref{eq:deformest_appl} and \eqref{eq:preimage_est} in the last step.

 \textit{Step 4:} We now move $T'$ back to $X$. For the composition $f\coloneqq \phi \circ \psi$ from Step 1, we find \begin{equation}\label{eq:hash1}
     f_\#T=\phi_\# (\partial V')=\phi_\#(\partial P)+\partial(\phi_\#R)\,. \end{equation}
We want to compare $f_\# T$ with $T$. To this end, we consider the (geodesic) homotopy current $H$ from Subsection \ref{subsec:2.3}. Lemma \ref{lem:trans}, (3), gives $$d(x,f(x))\leq Ls\,,$$ so \eqref{eq:mass1} implies the mass bound \begin{equation}
    \M(H_{\#}(\llbracket0,1\rrbracket\times T))\leq (k+1) Ls\max\{1,L\}^k\M(T)\,.\label{eq:mass1_s}
\end{equation}

Set $$Q\coloneqq\phi_\#(\partial P)\qquad \text{and}\qquad S\coloneqq \phi_\# R-H_{\#}(\llbracket0,1\rrbracket\times T)\,.$$ Both currents have compact support because $\Sigma$ is compact and $\phi$, $f$, and $H$ are continuous. Since $\partial^2 P=0$, the current $\phi_\#(\partial P)$ is a cycle. Therefore, we apply \eqref{eq:hash1} and \eqref{eq:hash2} to find $$\partial S=\phi_\# (\partial R)-(f_\# T-T)= T-Q\,,$$ which is nothing but \eqref{eq:decomp1}. 
Using the second inequalities in 
\eqref{eq:deformest_appl} and \eqref{eq:preimage_est}, we find \begin{equation}\label{eq:mass2}
    \M(R)\leq C_1s\M(T')\leq C_1 s L^k \M(T)\,.
\end{equation} Using  \eqref{eq:lip_bdd} once more gives $\M(\phi_\# R)\leq L^{k+1} \M(R)$, and the previous mass estimates \eqref{eq:mass1_s} and \eqref{eq:mass2} lead to
\begin{align*}
    \M(S) &\leq \M(\phi_\# R)+\M(H_{\#}(\llbracket0,1\rrbracket\times T))\notag\\ &\leq (L^{k+1} C_1 L^k +(k+1) L\max\{1,L\}^k) s\M(T)\eqqcolon a_0 s\M(T)\,,
\end{align*}
 proving \eqref{eq:decomp2}.

 \textit{Step 5:} Finally, we fill the boundaries of the top-dimensional simplices pushed to $X$. For any given $(k+1)$-simplex define the compactly supported integral $k$-cycle $$U_\sigma\coloneqq \phi_\# (\partial \llbracket \sigma\rrbracket)\,.$$ The boundary of a $(k+1)$-simplex consists of $k+2$ regular $k$-simplices. As each of the faces has mass $\omega_{k}s^k$ and restricts $\phi$ to an $L$-Lipschitz function, we observe that
 $$\M(U_\sigma)\leq (k+2)L^k \omega_k s^k\eqqcolon c_0 s^k\,.$$
 Note that $\F_X(U_\sigma)\leq \FV_{k+1}^X(c_0s^k)$ by definition. Applying Theorem \ref{thm:wen_ex} to $U_\sigma$ gives a minimizing filling $W_\sigma$ with $$\partial W_\sigma=U_\sigma\qquad \text{and}\qquad \M(W_\sigma)\leq \FV_{k+1}^X(c_0s^k)\,.$$ We set $$W\coloneqq \sum_\sigma n_\sigma W_\sigma\,.$$ As there are only finitely many simplices, $W$ is an integral current with compact support and $$\partial W=\sum_\sigma n_\sigma U_\sigma= \phi_\# \left(\partial \left(\sum_\sigma n_\sigma \llbracket \sigma\rrbracket\right)\right)= \phi_\# (\partial P)=Q\,.$$ Thus, $W$ is a filling of $Q$. Then \eqref{eq:sumnsigma} and the minimizing filling property imply \begin{align*}
     \M(W)&\leq \sum_\sigma |n_\sigma|\M(W_\sigma)\\&\leq \frac{C_1 L^{k+1}\F_X(T)}{\omega_{k+1}s^{k+1}} \FV_{k+1}^X(c_0s^k)\eqqcolon b_0\frac{\FV_{k+1}^X(c_0s^k)}{s^{k+1}} \F_X(T)\,, 
 \end{align*}
which proves \eqref{eq:decomp3} and thus concludes the proof of the proposition.

\begin{remark}\label{rmk}
Note that all constants $a_0,b_0,c_0,s_0$ depend only on $k$ and the $(n,c_1,s_0)$-Nagata cover as $$a_0= C_1 L^{2k+1} +(k+1) L\max\{1,L\}^k\,,\quad b_0= \frac{C_1 L^{k+1}}{\omega_{k+1}}\,, \quad c_0=(k+2)L^k \omega_k $$
with $L=L(n,c_1)$ the Lipschitz constant from Lemma \ref{lem:trans}, $C_1=C_1(n)$ the constant from the deformation Theorem \ref{thm:deform}, and $\omega_m$ the volume of the regular, unit-edge $m$-simplex, $m\in \{k,k+1\}$. 
\end{remark}

\section{Proof of the linear filling inequality} \label{sec:5}
In this section, we give a proof of Theorem \ref{thm:1} and \ref{thm:2}. Since Proposition~\ref{prp:1} uses compactly supported currents and the main theorems are stated for currents of finite mass, we need the following lemma, which promotes a linear filling inequality for compactly supported currents to a linear filling inequality for currents of finite mass, up to a change of constant.
\begin{lemma}[Promotion to integral currents of finite mass]\label{lem:promo} Let $X$ be a complete $\mathrm{CAT}(0)$-space, and let $k\geq 1$. If there is a $C>0$ such that every $S\in \I_{k,c}(X)$ with $\partial S=0$ admits a filling $V\in \I_{k+1,c}(X)$ with $\partial V=S$ and
\begin{equation}
\label{eq:lin_fill_ass}
\M(V)\leq C\M(S)\,,
\end{equation} then there is a $C'=C'(C,k)>0$ such that every $T\in \I_k(X)$ with $\partial T=0$ satisfies
\begin{equation}
    \label{eq:lin_fill_claim}
    \F_X(T)\leq C' \M(T)\,.
\end{equation}
\end{lemma}
The proof of this lemma is based on Wenger's thick-thin decomposition. Note that it does not rely on a finite asymptotic Nagata dimension and might be of independent interest. To keep focus on the proof of the main theorems, the proof of Lemma \ref{lem:promo} is postponed to Appendix \ref{app}.

Now we are in position to prove our main results.

\begin{proof}[Proof of Theorem \ref{thm:1}]
    We apply Proposition \ref{prp:1}. This gives $a_0,b_0,c_0,s_0>0$ such that every cycle $T\in \I_{k,c}(X)$ has a decomposition  $$T=Q+\partial S$$ with  $S\in \I_{k+1,c}(X)$ and $Q\in \I_{k,c}(X)$ and the properties \begin{equation}
        \M(S)\leq a_0 s \M(T)\qquad \text{and}\qquad
        \M(W)\leq b_0 \frac{\FV_{k+1}^X(c_0 s^k)}{s^{k+1}}  \F_X(T)\label{eq:assumpt}
    \end{equation} for a filling $W\in \I_{k+1,c}(X)$ of $Q$  and for all $s\geq s_0$. Since $\partial (S+W)=T$, we obtain by triangle inequality and \eqref{eq:assumpt} that
    \begin{equation}
        \label{eq:subaddbdd}
   \F_X(T)\leq \M(S+W)\leq \M(S)+\M(W)\leq a_0 s \M(T)+\alpha(s)  \F_X(T) \end{equation} with $\alpha(s)= b_0 s^{-(k+1)}\FV_{k+1}^X(c_0 s^k)$. The quantity $\alpha(s)$ tends to $0$ as $s\to\infty$ by assumption \eqref{eq:subEucl}. Fix $s_*\geq s_0$ such that $\alpha(s_*)\leq 1/2$. Then \eqref{eq:subaddbdd} gives \begin{equation*}       \F_X(T)\leq 2a_0 s_* \M(T)\,.\end{equation*} Applying Lemma \ref{lem:promo}, we can extend this linear filling inequality to cycles of finite mass, and thus complete the proof.
\end{proof} We stress that the final constant $2a_0s_*$ does not depend on $T$. Compare the inequality \eqref{eq:subaddbdd} given by $$ \F_X(T)\leq a_0 s \M(T)+\alpha(s)  \F_X(T)$$ with the previous control from \cite[Theorem 1.1]{LSU26} given by $$\F_X(T)\leq C_\delta \M(T)^{1+\delta}\,.$$  The absorbable filling term in the former estimate is the analytic gain of our approach.

As a consequence of Theorem \ref{thm:1}, we obtain a proof of Gromov's linear isoperimetric filling inequality for finite asymptotic Nagata dimension.

\begin{proof}[Proof of Theorem \ref{thm:2}]
    A complete $\mathrm{CAT}(0)$-space is geodesic (and thus $1$-quasiconvex) and satisfies cone-type inequalities in all dimensions by Theorem \ref{thm:wen_coning}. If $k\geq \max\{\nu, 1\}$, then Theorem \ref{thm:wen_subeuc} implies that \eqref{eq:subEucl} holds, and thus we can conclude by applying Theorem \ref{thm:1}. 
\end{proof}

\section{Applications of the linear filling inequality}\label{sec:6}
 In this section we discuss three immediate corollaries.
 
 Together with Wenger's Euclidean growth theorem \cite[Theorem 1.4]{Wen11} for $k<\nu$, Theorem \ref{thm:2} describes a jump in the growth rates of the filling function, which is a classical, group-theoretic invariant for behavior at large scales; see \cite{Gro93}. Put differently, the following characterization of the asymptotic rank in terms of the filling function holds.
\begin{corollary}[Another characterization of asymptotic rank]
  Let $X$ be a complete  $\mathrm{CAT}(0)$-space with finite asymptotic Nagata dimension and finite asymptotic rank $\nu\geq 1$. Then $$\nu=\min\{k\geq 1: \FV^X_{k+1}(r)=\mathcal  O(r)\}\,.$$
\end{corollary}

\begin{proof}
     Theorem \ref{thm:2} in the form \eqref{eq:FV_asymp} shows that $k=\nu$ is admissible for the minimum. Then \cite[Theorem 1.4]{Wen11} shows that the minimum is at least $\nu$. This concludes the proof.
\end{proof}

Hence, for the given class of spaces, the asymptotic rank coincides with the first dimension in which the filling function exhibits linear growth. This corollary emphasizes the phase transition in filling behavior that occurs when the dimension of a $k$-cycle crosses the threshold of the asymptotic rank. 
\vspace{0.2cm}

 Our next consequence is a complementary positivity result for Cheeger's constant. Given a Hadamard manifold $M$ with uniformly bounded sectional curvature from above by $-K<0$, Yau \cite{Yau75} established a bound for \textit{Cheeger's constant} \cite{Che70} of the form 
 $$h(M)\coloneqq \inf_E \frac{\mathrm{Per}(E)}{\mathrm{vol}(E)}\geq (n-1)\sqrt K\,,$$ where the infimum is taken over all bounded $E$ of finite perimeter. The volume is denoted $\mathrm{vol}(E)$ and the perimeter $\mathrm{Per}(E)$. There is also an intimate relation between positivity of the Cheeger constant for normal covers of closed manifolds and the non-amenability of the deck group; see \cite{Bro81}. We also mention that, more recently, flow techniques have been used to establish a Cheeger-type bound for quotient spaces under pinched negative curvature; see \cite{CMW23}. In contrast, our corollary neither makes any assumptions on the uniformity of negative curvature nor on group actions, and it works for finite asymptotic rank and finite asymptotic Nagata dimension.

\begin{corollary}
    [Positive Cheeger constant]
    Let $M$ be an oriented Hadamard manifold of dimension $n\geq 2$ with finite asymptotic Nagata dimension and asymptotic rank at most $n-1$. Then there exists a constant $C>0$ such that every bounded finite-perimeter set $E\subseteq M$ satisfies \begin{equation}
        \label{eq:isop_chee}
    \mathrm{vol}(E)\leq C \mathrm{Per}(E)\,.\end{equation}  In particular, the associated Cheeger constant $h(M)$ and the bottom $\lambda(M)$ of the spectrum of the Laplace--Beltrami operator satisfy
    $$\lambda (M)\geq \frac{h(M)^2}{4}\geq \frac1{4C^2}\,.$$
\end{corollary}
Since Hadamard $3$-manifolds have finite Nagata dimension $3$ (see \cite{JL22}), this immediately implies a positivity result for asymptotic rank at most $2$.

\begin{proof}Recall that in case $M$ is an oriented, $n$-dimensional manifold and $E\subseteq M$ is bounded and of finite perimeter, then $$\M(\llbracket E\rrbracket)=\mathrm{vol}(E)\qquad \text{and} \qquad \M(\partial \llbracket E\rrbracket)=\mathrm{Per}(E)\,,$$ which follows from \eqref{eq:E_form} and \cite[Proposition 2.15]{DeL16} applied locally in charts, for instance. Hence,  \eqref{eq:isop_chee} is just the linear filling inequality for $T=\partial \llbracket E\rrbracket$ and $k=n-1$. Indeed, the minimizing filling $V$ of $\partial \llbracket E\rrbracket$  exists due to \cite[Theorem 1.6]{Wen05} and is uniquely given by $\llbracket E\rrbracket$ due to the constancy theorem \cite{FF60,Fed69}. For the latter it suffices that the $n$-cycle $V-\llbracket E\rrbracket$ for some minimizing filling $V$ is a compactly supported cycle in the proper, connected, oriented manifold $M$. Note that $V$ is compactly supported by Theorem \ref{thm:wen_ex}. The last part of the corollary follows by Cheeger's inequality and the definition of Cheeger's constant \cite{Che70}.
\end{proof}

\vspace{0.2cm}

Another interesting application concerns the deformations of relative cycles as trea\-ted in \cite{BWY23}.
Indeed, under the additional assumptions of \cite[Theorem 5.1]{BWY23}, we can deduce the following improved mass bound for their linear deformation result.

We adapt the notions from \cite[Section 1]{BWY23}. We say that $X$ is \textit{$k$-Lipschitz connected} for some $k\geq 0$, abbreviated $(LC_k)$, if there is a $c\geq 1$ such that for any $l\leq k$ every $L$-Lipschitz map from the Euclidean $l$-sphere to $X$ extends to a $cL$-Lipschitz map on the $(l+1)$-ball. We say that $X$ satisfies Euclidean isoperimetric inequalities for some $k\geq 0$, abbreviated $(EI_k)$, if there exists a $C>0$ such that for every $m\in\{1,\dots ,k\}$ and every cycle $T\in \I_m(X)$ we have
\begin{equation}
    \label{eq:Euclineq}
\F_X(T)\leq C \M(T)^{\frac {m+1}m}\,.\end{equation}
In particular, for $k=0$, the condition is empty and thus trivially satisfied.

\begin{corollary}
    [Linear relative deformation]
    Let $X$ be a complete  $\mathrm{CAT}(0)$-space with finite asymptotic Nagata dimension and finite asymptotic rank $\nu$. Let $k+1\geq \nu$, and let $K\subseteq X$ be compact, $q$-quasiconvex, and either $(LC_k)$ or $(EI_k)$ for some $k\geq 0$. For every $S\in \I_{k+1,c}(X)$ with $\supp(\partial S)\subseteq K$ of finite Nagata dimension, there are $C>0$, $\bar S\in \I_{k+1,c}(X)$, and $W\in \I_{k+2,c}(X)$ satisfying $\partial W=\bar S-S$, $\partial \bar S=\partial S$, $\supp\bar S\subseteq K$, and \begin{equation}
        \M(\bar S)+\M(W)\leq C \M(S)\,,\label{eq:mass_deform}
    \end{equation} where $C$ depends only on $q$, $k$, the linear filling constant, the Nagata data of $\supp(\partial S)$, and the data of either $(LC_k)$ or $(EI_k)$.
\end{corollary} 

\begin{proof} If $\partial S=0$, we can take $\bar S=0$ and construct a filling $W$ of $-S$ by Theorem~\ref{thm:2} directly with the same constant in \eqref{eq:mass_deform} as in \eqref{eq:linfil}.
For the case $\partial S\neq 0$, apply the first part of \cite[Theorem 5.1]{BWY23} to $\supp (\partial S)\subseteq K\subseteq X$, which leads to an $\bar S\in \I_{k+1,c}(X)$ supported in $K$ such that $$\partial \bar S=\partial S\qquad \text{and}\qquad \M (\bar S)\leq C_1 \M(S)$$ for some $C_1>0$. Thus, Theorem \ref{thm:2} gives a $W\in \I_{k+2,c}(X)$ with $$\partial W=\bar S- S\qquad \text{and} \qquad \M(W)\leq  C_2 \M (\bar S-S)\leq (C_1+1)C_2\M(S)$$ for some constant $C_2>0$, which concludes the proof with $C=C_1+(C_1+1)C_2$.
\end{proof}

We emphasize that we do not make use of the last part of \cite[Theorem 5.1]{BWY23} and thus obtain an independent mass bound for $\M(W)$ without assuming $X$ being a Banach space. Further note that this bound does not depend on the support of $S$.

\appendix
\section{Thick-thin decomposition}\label{app}

The thick-thin decomposition is a versatile tool developed in \cite[Theorem 4.2]{Wen11a} and used in \cite[Theorem 7.2]{Wen11} to pass from currents of bounded support to currents of finite mass. We go one step further and start with compactly supported currents. 

Instead of repeating the thick-thin decomposition from \cite{Wen11a} in its full generality, we apply it in the next lemma, which records the consequences that we need for the proof of Lemma \ref{lem:promo}.

\begin{theorem}[Application of {\cite[Theorem 4.2]{Wen11a}}]\label{thm:wen_thick}
  Let $X$ be a complete $\mathrm{CAT}(0)$-space. Then there is a $\gamma\in (0,1)$ such that for every $T\in \I_k(X)$, $T\neq 0$, with $\partial T=0$ there are $R\in \I_k(X)$ and $T_j\in \I_k(X)$, $j\in \N$, such that 
  \begin{equation*}
   T=R+\sum_{j=1}^\infty T_j
\end{equation*}
and the following properties hold:

\begin{enumerate}
   \item[(i)] It holds that \begin{equation}\label{eq:boundary}
    \partial R=\partial T=0\qquad \text{and}\qquad \partial T_j=0\,,\quad j\in \N\,.
\end{equation} 
   \item[(ii)] There is a $k$-dependent constant $D_1>0$ such that for all $x\in \supp R$ and $0\leq  r\leq 5\min\{1/2,\M(T)^{-1/k}\} \M(R)^{1/k}$, it holds that \begin{equation}\label{eq:low_bdd}
        \|R\|(B_r(x))\geq D_1 \gamma r^k\,.
    \end{equation}
   \item[(iii)]For all $j\in \N$, we have the mass estimate \begin{equation}\label{eq:T_j_bdd}\M(T_j)\leq 
    \frac 32 
    \gamma\,.
    \end{equation}
   \item[(iv)]There is a $k$-dependent constant $D_2>0$ such that \begin{equation}
       \mathrm{diam}(\supp T_j)\leq D_2\gamma^{-\frac 1k}\M(T_j)^{\frac 1k}\,.\label{eq:diam}\end{equation}
   \item[(v)] Finally, we have the mass estimate \begin{equation}\label{eq:mass_sum}\M(R)+\frac 13 \sum_{j=1}^\infty\M(T_j)\leq \M(T)\,.
   \end{equation}
\end{enumerate}
\end{theorem}
Here $R$ is the ``thick'' part and $T_j$ the ``thin'' parts of the decomposition of $T$. The lemma uses $\partial T=0$ and the additional structure of a complete $\mathrm{CAT}(0)$-space.

\begin{proof}

 In the language of \cite[Theorem 4.2]{Wen11a}, we take $\alpha=0$ for $k=1$ and $\alpha=k$ for $k\geq 2$. The former case does not require any further prerequisites to be met. If $k\geq 2$, then Hadamard spaces satisfy a Euclidean-type filling inequality of the form \eqref{eq:Euclineq} for $m=k-1$ by \cite[Theorem 1.2]{Wen05}, which uses Theorem \ref{thm:wen_coning}.

Therefore, \cite[Theorem 4.2]{Wen11a} is applicable and provides the $\gamma\in (0,1)$. Let $T\in \I_k(X)$, $T\neq 0$, with $\partial T=0$. Set $\lambda=1/2$, $F(r)=\gamma r^k$, $G(r)=r^{1/k}$, and $\delta=\min\{1/2,\M(T)^{-1/k}\}$. Then \cite[Theorem 4.2]{Wen11a} provides $R$ and $T_j$ as stated; we verify the five consequences one by one. 

For (i) we just use $\partial T=0$. For (ii) we note that $\supp(\partial T)=\emptyset$ and that $D_1=5^{-k-\alpha}/2$. For (iii) we bound $\delta^k\M(T)\leq 1$ by definition of $\delta$. For (iv) we set 
$D_2=4(4\cdot 5^{k+\alpha})^{1/k}$. Finally, for (v) we just insert the definition of $\lambda$. Note that $D_1$ and $D_2$ depend on $k$ only as $\alpha$ is determined by $k$.  
\end{proof}

As a last input for the proof of the promotion lemma, we will need the following basic covering argument.
\begin{lemma}\label{lem:tot_bdd} Let $X$ be a metric space, and let $k\geq 1$. Let $B_r(x)$ denote the ball at $x\in X$ of radius $r>0$ and $\mu$ be a finite Borel measure on $X$. Assume there are $r_0,\mu_0>0$ such that \begin{equation}\label{eq:low_bdd_mu}
    \mu(B_r(x))\geq \mu_0 r^k
\end{equation} for every $x\in \supp \mu$ and $0<r\leq r_0$. Then $\supp \mu$ is totally bounded.
\end{lemma}

\begin{proof}
    Let $0<r\leq r_0$. If $\supp \mu$ could not be covered by finitely many balls of radius $r$, then there are countably many $x_i\in \supp \mu$ satisfying $$d(x_i,x_j)\geq r\,,\qquad i\neq j\,.$$
    As the balls $B_{r/3}(x_i)$ are consequently pairwise disjoint, \eqref{eq:low_bdd_mu} implies $$\mu(X)\geq \sum_{i=1}^N \mu(B_{r/3}(x_i))\geq N\mu_0 \frac{r^k}{3^k}$$ for every $N\in \N$, which implies that $\mu(X)=\infty$ as $N\to \infty$, a contradiction. If $r>r_0$, any finite cover by balls of radius $r_0$ is also a cover by balls of radius $r$. Hence, $\supp \mu$ has a finite cover by balls of radius $r$, so it is totally bounded by definition.
    \end{proof}

\begin{proof}[Proof of Lemma \ref{lem:promo}] 
Let $T\in \I_k(X)$ be a cycle. We may assume that $\M(T)>0$ since otherwise $T=0$ and there is nothing to prove. Applying Theorem \ref{thm:wen_thick}, we obtain $R,T_j\in \I_k(X)$ and $D_1,D_2>0$ with the properties stated in the theorem.

We first show that $R$ has compact support. To this end, we can assume that $\M(R)>0$ since otherwise $R$ is already compactly supported. Thus, we know that the quantity $r_R\coloneqq 5\min\{1/2,\M(T)^{-1/k}\} \M(R)^{1/k}$ is positive. Then $\supp R$ is totally bounded by Lemma \ref{lem:tot_bdd} with $\mu=\|R\|$, $\mu_0=D_1\gamma$, and $r_0=r_R$ using \eqref{eq:low_bdd}. Since $\supp R$ is also closed in the complete metric space $X$, it is therefore compact. As $R$ is compactly supported, we can apply the assumption \eqref{eq:lin_fill_ass} to find a $U\in \I_{k+1,c}(X)$ with $\partial U=R$ and $$\M(U)\leq C \M(R)\leq C\M(T)\,,$$ where we used \eqref{eq:mass_sum} in the last step.

After having filled $R$, we fill the thin parts $T_j$ as well. Using Theorem~\ref{thm:wen_coning} with \eqref{eq:boundary} and \eqref{eq:diam}, there exists a filling $V_j\in \I_{k+1}(X)$ of $T_j$ such that \begin{equation*}
    \M(V_j)\leq \frac{D_2\gamma^{-\frac 1k}}{k+1} \M(T_j)^{1+\frac 1k}\leq\left( \frac 32 \right)^{\frac 1k}\frac{D_2 }{k+1}\M(T_j)\eqqcolon D_3 \M(T_j)\,,
\end{equation*} where we used \eqref{eq:T_j_bdd} in the last step.
Summing over $j$ and applying \eqref{eq:mass_sum} yields \begin{equation}
    \sum_{j=1}^\infty \M(V_j)\leq 3 D_3\M(T) \quad \text{and} \quad \sum_{j=1}^\infty \M(\partial V_j)=\sum_{j=1}^\infty \M(T_j) \leq 3 \M(T)\label{eq:series}
\end{equation} 
Since the space of currents of finite mass is complete with respect to $\M$ (see \cite[p.~15]{BWY23}, for instance), the partial sums $\sum_j V_j$ of the series in \eqref{eq:series} converge in the mass norm to a current $V_\infty$ of finite mass and $$\partial V_\infty=\sum_{j=1}^\infty T_j=T-R\,.$$ As the partial sums are integral currents and their masses and boundary masses are uniformly bounded by \eqref{eq:series}, we obtain that $V_\infty\in \I_{k+1}(X)$ with $\M(V_\infty)\leq 3 D_3 \M(T)$ by the closure theorem \cite[Theorem 8.5]{AK00}. Finally, setting $$V\coloneqq U+V_{\infty}\,,$$ we obtain $\partial V=T$ and $$\M(V)\leq (C+3D_3)\M(T)\,,$$ which proves the claim \eqref{eq:lin_fill_claim} with $C'=C+3D_3$.
\end{proof}



\begin{thebibliography}{DLPS25}

\bibitem[Alm86]{Alm86}
F.~J.~Almgren, Jr.,
\emph{Optimal isoperimetric inequalities},
Indiana Univ. Math. J. \textbf{35} (1986), no.~3, 451--547.


\bibitem[AK00]{AK00}
L.~Ambrosio and B.~Kirchheim,
\emph{Currents in metric spaces},
Acta Math. \textbf{185} (2000), 1--80.

\bibitem[Ass82]{Ass82}
P.~Assouad,
\emph{Sur la distance de Nagata},
C. R. Acad. Sci. Paris S\'er. I Math. \textbf{294} (1982), no.~1, 31--34.

\bibitem[BE99]{BE99}
H.~Bahn and P.~Ehrlich,
\emph{A Brunn--Minkowski type theorem on the Minkowski spacetime},
Canad. J. Math. \textbf{51} (1999), no.~3, 449--469.

\bibitem[Bal95]{Bal95}
W.~Ballmann,
\emph{Lectures on Spaces of Nonpositive Curvature},
DMV Seminar \textbf{25}, Birkh\"auser, 1995.


\bibitem[BWY23]{BWY23}
G.~Basso, S.~Wenger, and R.~Young,
\emph{Undistorted fillings in subsets of metric spaces},
Adv. Math. \textbf{423} (2023), 109024.

\bibitem[BH99]{BH99}
M.~R.~Bridson and A.~Haefliger,
\emph{Metric Spaces of Non-Positive Curvature},
Grundlehren der mathematischen Wissenschaften \textbf{319}, Springer, 1999.

\bibitem[Bro81]{Bro81}
R.~Brooks,
\emph{The fundamental group and the spectrum of the Laplacian},
Comment. Math. Helv. \textbf{56} (1981), 581--598.

\bibitem[BBI01]{BBI01}
D.~Burago, Y.~Burago, and S.~Ivanov,
\emph{A Course in Metric Geometry},
Graduate Studies in Mathematics \textbf{33}, American Mathematical Society, 2001.

\bibitem[Bus55]{Bus55}
H.~Busemann,
\emph{The Geometry of Geodesics},
Academic Press, New York, 1955.

\bibitem[Bus82]{Bus82}
P.~Buser,
\emph{A note on the isoperimetric constant},
Ann. Sci. \'Ecole Norm. Sup. (4) \textbf{15} (1982), 213--230.

\bibitem[CS07]{CS07}
L.~Caffarelli and L.~Silvestre,
\emph{An extension problem related to the fractional Laplacian},
Commun. Partial Differ. Equ. \textbf{32} (2007),
no.~8, 1245--1260.

\bibitem[CM17]{CM17}
F.~Cavalletti and A.~Mondino,
\emph{Sharp and rigid isoperimetric inequalities in metric-measure spaces with lower Ricci curvature bounds},
Invent. Math. \textbf{208} (2017), 803--849.

\bibitem[CM25]{CM25}
F.~Cavalletti and A.~Mondino,
\emph{A sharp isoperimetric-type inequality for Lorentzian spaces satisfying timelike Ricci lower bounds},
arXiv:2401.03949 (2024).

\bibitem[Che70]{Che70}
J.~Cheeger,
\emph{A lower bound for the smallest eigenvalue of the Laplacian}, In: Problems in Analysis (Papers Dedicated to Salomon Bochner, 1969), Princeton Univ. Press, Princeton, NJ, 1970, 195--199.

\bibitem[CMW23]{CMW23}
C.~Connell, D.~B.~McReynolds, and S.~Wang,
\emph{The natural flow and the critical exponent},
arXiv:2302.12665 (2023).

\bibitem[DeL16]{DeL16}
C.~De Lellis,
\emph{The size of the singular set of area-minimizing currents}, Surv. Differ. Geom. \textbf{21} (2016), 1--83.

\bibitem[Dod84]{Dod84}
J.~Dodziuk,
\emph{Difference equations, isoperimetric inequality and transience of certain random walks},
Trans. Amer. Math. Soc. \textbf{284} (1984), 787--794.

\bibitem[DS07]{DS07}
A.~N.~Dranishnikov and J.~Smith,
\emph{On asymptotic Assouad--Nagata dimension},
Topology Appl. \textbf{154} (2007), no.~4, 934--952.

\bibitem[DLPS25]{DLPS25}
C.~Dru\c{t}u, U.~Lang, P.~Papasoglu, and S.~Stadler,
\emph{Minimal tetrahedra and an isoperimetric gap theorem in non-positive curvature},
arXiv:2502.03389 (2025).

\bibitem[Fed69]{Fed69}
H.~Federer,
\emph{Geometric Measure Theory},
Grundlehren der mathematischen Wissenschaften \textbf{153},
Springer, Berlin, 1969.

\bibitem[FF60]{FF60}
H.~Federer and W.~H.~Fleming,
\emph{Normal and integral currents},
Ann. of Math. \textbf{72} (1960), no.~2, 458--520.

\bibitem[GL23]{GL23}
T.~Goldhirsch and U.~Lang,
\emph{Characterizations of higher rank hyperbolicity},
Math. Z. \textbf{305} (2023), no.~13.

\bibitem[Gro83]{Gro83}
M.~Gromov,
\emph{Filling Riemannian manifolds},
J. Differ. Geom. \textbf{18} (1983), 1--147.

\bibitem[Gro87]{Gro87}
M.~Gromov,
\emph{Hyperbolic groups},
In: S.~Gersten (ed.), Essays in Group Theory,
Math. Sci. Res. Inst. Publ. \textbf{8}, Springer, 1987, 75--263.

\bibitem[Gro93]{Gro93}
M.~Gromov,
\emph{Asymptotic invariants of infinite groups},
In: A.~Niblo and M.~A.~Roller (eds.), Geometric Group Theory,
London Math. Soc. Lecture Note Ser. \textbf{182},
Cambridge Univ. Press, 1993, 1--295.

\bibitem[Isl25]{Isl25}
H.~Isleifsson,
\emph{Linear isoperimetric inequality for homogeneous Hadamard manifolds},
J. Topol. Anal. \textbf{17} (2025), no.~3, 761--768.

\bibitem[JL22]{JL22}
M.~J{\o}rgensen and U.~Lang,
\emph{Geodesic spaces of low Nagata dimension},
Ann. Fenn. Math. \textbf{47} (2022), no.~1, 83--88.

\bibitem[Kle99]{Kle99}
B.~Kleiner,
\emph{The local structure of length spaces with curvature bounded above},
Math. Z. \textbf{231} (1999), 409--456.

\bibitem[KL20]{KL20}
B.~Kleiner and U.~Lang,
\emph{Higher rank hyperbolicity},
Invent. Math. \textbf{221} (2020), no.~2, 597--664.


\bibitem[Lan00]{Lan00}
U.~Lang,
\emph{Higher-dimensional linear isoperimetric inequalities in hyperbolic groups}, Int. Math. Res. Not. (2000), no.~13, 709--717.

\bibitem[Lan11]{Lan11}
U.~Lang,
\emph{Local currents in metric spaces},
J. Geom. Anal. \textbf{21} (2011), 683--742.


\bibitem[LS05]{LS05}
U.~Lang and T.~Schlichenmaier,
\emph{Nagata dimension, quasisymmetric embeddings, and Lipschitz extensions},
Int. Math. Res. Not. (2005), no.~58, 3625--3655.

\bibitem[LSU26]{LSU26}
U.~Lang, S.~Stadler, and D.~Urech,
\emph{Isoperimetric inequalities in Hadamard spaces of asymptotic rank two},
Adv. Math. \textbf{502} (2026), 111123.

\bibitem[LW08]{LW08}
U.~Lang and S.~Wenger,
\emph{Isoperimetric inequalities and the asymptotic geometry of Hadamard spaces}, Enseign. Math. (2) \textbf{54} (2008), no.~1--2, 135--137.

\bibitem[LP25]{LP25}
C.~Lange and J.~W.~Peteranderl,
\emph{Quantitative Lorentzian isoperimetric inequalities in conical Minkowski spacetimes},
arXiv:2510.26755 (2025).

\bibitem[Leu14]{Leu14}
E.~Leuzinger,
\emph{Optimal higher-dimensional Dehn functions for some CAT(0) lattices},
Groups Geom. Dyn. \textbf{8} (2014), no.~2, 441--466. 

\bibitem[Mil65]{Mil65}
J.~Milnor,
\emph{Lectures on the $h$-Cobordism Theorem},
Notes by L.~Siebenmann and J.~Sondow,
Princeton University Press, Princeton, NJ, 1965.

\bibitem[Oss78]{Oss78}
R.~Osserman,
\emph{The isoperimetric inequality},
Bull. Amer. Math. Soc. \textbf{84} (1978), 1182--1238.

\bibitem[Sch20]{Sch20}
F.~Schulze,
\emph{Optimal isoperimetric inequalities for surfaces in any codimension in Cartan--Hadamard manifolds},
Geom. Funct. Anal. \textbf{30} (2020), no.~1, 255--288.

\bibitem[Sta21]{Sta21}
S.~Stadler,
\emph{The structure of minimal surfaces in $\operatorname{CAT}(0)$ spaces},
J. Eur. Math. Soc. \textbf{23} (2021), no.~11, 3521--3554.

\bibitem[Wen05]{Wen05}
S.~Wenger,
\emph{Isoperimetric inequalities of Euclidean type in metric spaces},
Geom. Funct. Anal. \textbf{15} (2005), 534--554.

\bibitem[Wen06]{Wen06}
S.~Wenger,
\emph{Filling invariants at infinity and the Euclidean rank of Hadamard spaces},
Int. Math. Res. Not. (2006), 83090.

\bibitem[Wen08]{Wen08}
S.~Wenger,
\emph{Gromov hyperbolic spaces and the sharp isoperimetric constant},
Invent. Math. \textbf{171} (2008), 227--255.

\bibitem[Wen11]{Wen11}
S.~Wenger,
\emph{The asymptotic rank of metric spaces},
Comment. Math. Helv. \textbf{86} (2011), 247--275.

\bibitem[Wen11a]{Wen11a}
S.~Wenger,
\emph{Compactness for manifolds and integral currents with bounded diameter and volume},
Calc. Var. Partial Differential Equations \textbf{40} (2011),
no.~3--4, 423--448.



\bibitem[Yau75]{Yau75}
S.-T.~Yau,
\emph{Isoperimetric constants and the first eigenvalue of a compact Riemannian manifold},
Ann. Sci. \'Ecole Norm. Sup. \textbf{8} (1975), no.~4, 487--507.
\end{thebibliography}
\end{document}